\documentclass[11pt,a4paper,reqno]{amsart}
\usepackage{amsmath,amssymb,amsfonts,epsfig,mathrsfs,cite}
\usepackage[T1]{fontenc}
\usepackage{color}
\usepackage{array}
\usepackage{amsthm}
\usepackage{amstext}
\usepackage{graphicx}
\usepackage{setspace}

\usepackage{hyperref}

\numberwithin{equation}{section}

\usepackage{amsfonts}
\usepackage{amsmath}
\usepackage{amssymb}
\usepackage{shuffle}
\usepackage{graphicx}
\usepackage[official]{eurosym}
\usepackage{cancel}
\usepackage{url}
 \DeclareFontFamily{U}{shuffle}{}
 \DeclareFontShape{U}{shuffle}{m}{n}{%
 <5-8>shuffle7%
 <8->shuffle10%
 }{}

 \DeclareSymbolFont{Shuffle}{U}{shuffle}{m}{n}

\DeclareMathSymbol\shuffle{\mathbin}{Shuffle}{"001}
\DeclareMathSymbol\cshuffle{\mathbin}{Shuffle}{"002}

\usepackage[title]{appendix}

\makeatletter
\@namedef{subjclassname@2020}{%
  \textup{2020} Mathematics Subject Classification}
\makeatother

\newtheorem{theorem}{Theorem}[section]

\usepackage{amscd,psfrag}
\usepackage{yhmath}
\usepackage[mathscr]{eucal}

\newtheorem{lemma}{Lemma}[section]

\newtheorem{remark}{Remark}[section]

\usepackage{comment}

\newcommand{\R}{\mathbb{R}}

\newcommand{\na}{\nabla}

\newcommand{\p}{\partial}

\newcommand{\e}{\varepsilon}

\newcommand{\dd}{{\rm d}}

\newcommand{\G}{\Gamma}

\newcommand{\two}{{\rm II}}

\newcommand{\1}{{\mathbf{1}}}

\newcommand{\RR}{{\mathcal{R}}}

\def\XXint#1#2#3{{\setbox0=\hbox{$#1{#2#3}{\int}$ }
\vcenter{\hbox{$#2#3$ }}\kern-.6\wd0}}

\newcommand{\pbad}{\wp}
\newcommand{\QQ}{{\mathcal{Q}}}
\newcommand{\ppsi}{{\psi_{x_1, x_2}}}
\newcommand{\ed}{{\eta^{\delta}}}
\newcommand{\pd}{{\varphi^{\delta}}}
\newcommand{\dist}{{\rm dist}}
\newcommand{\po}{{\p\Omega}}
\newcommand{\nor}{{\widetilde{\nu}}}
\newcommand{\dduu}{{{\rm div}\,{\rm div}\,(u\otimes u)}}

\newcommand{\kn}{{\mathcal{K}_N}}

\newcommand{\group}{{\left(\ed\ppsi\right)}}

\newcommand{\td}{{\widetilde{\bf \delta}}}

\newcommand{\tilo}{{\widetilde{\Omega}}}

\usepackage{xcolor}

\usepackage{comment}

\allowdisplaybreaks[4]

\usepackage{slashed}

\title{A note for double Hölder regularity of the hydrodynamic pressure for weak solutions of Euler equations}

\author{Siran Li: School of Mathematical Sciences and CMA-Shanghai, Shanghai Jiao Tong University, Shanghai, China (200240)}

\address{Siran Li: School of Mathematical Sciences and CMA-Shanghai, Shanghai Jiao Tong University, Shanghai, China (200240)}

\email{\texttt{siran.li@sjtu.edu.cn}}

\author{Ya-Guang Wang:  School of Mathematical Sciences, Center for Applied Mathematics, MOE-LSC, and SHL-MAC, Shanghai Jiao Tong University, Shanghai, China (200240)}

\address{Ya-Guang Wang:  School of Mathematical Sciences, Center for Applied Mathematics, MOE-LSC, and SHL-MAC, Shanghai Jiao Tong University, Shanghai, China (200240)}

\email{\texttt{ygwang@sjtu.edu.cn}}

\date{\today}

\begin{document}

\begin{abstract}
We give an elementary proof for the interior double H\"{o}lder regularity of the hydrodynamic pressure for weak solutions of the Euler Equations in a bounded $C^2$-domain $\Omega \subset \R^d$; $d\geq 3$. That is, for velocity $u \in C^{0,\gamma}(\Omega;\R^d)$ with some $0<\gamma<1/2$, we show that the pressure $p \in C^{0,2\gamma}_{\rm int}(\Omega)$. This is motivated by the studies of turbulence and anomalous dissipation in mathematical hydrodynamics and, recently, has been  established in [L. De Rosa, M. Latocca, and G. Stefani, \textit{Int. Math. Res. Not.} \textbf{2024.3} (2024), 2511--2560] over $C^{2,1}$-domains by means of pseudodifferential calculus. Our approach involves only standard elliptic PDE techniques, and relies  on a variant of the modified pressure introduced in [C.~W. Bardos, D.~W. Boutros, and E.~S. Titi, H\"{o}lder regularity of the pressure for weak solutions of the 3D Euler equations in bounded domains,  \textit{Arch. Rational Mech. Anal.} \textbf{249} (2025), 28] and the potential estimates in [L. Silvestre, \textit{unpublished notes}]. The key novel ingredient of our proof is the introduction of two cutoff functions whose localisation parameters are carefully chosen as a power of the distance to $\po$.
\end{abstract}

\maketitle

\section{Introduction}
\label{sec: intro}

We are concerned with the Euler equations for the motion of incompressible inviscid fluids in a bounded $C^2$-domain $\Omega\subset\R^d$ with $d \geq 3$:
\begin{equation}\label{euler}
\begin{cases}
\p_t u + {\rm div}( u\otimes u) + \na p = 0\quad\text{and}\quad {\rm div}\, u = 0\quad \text{in } [0,T]\times \Omega,\\
u\cdot\nu = 0 \qquad\text{on }[0,T]\times \po.
\end{cases}
\end{equation}
Here, $u:[0,T]\times \Omega\to\R^d$ is the velocity and $p:[0,T]\times \Omega\to\R$ is the hydrodynamic pressure of the flow, and $\nu:\po\to\R^d$ is the outward normal vector field to the boundary $\po$. On $\po$ the impermeability boundary condition is imposed. The goal of this note is to give an alternative, elementary proof for the interior double H\"{o}lder regularity of $p$, and to provide refined characterisations for the limiting behaviour when approaching the boundary.

\subsection{Main Theorem}
The main result of the paper is as follows:
\begin{theorem}\label{thm}
Let $(u,p)$ be a weak solution of the Euler equation~\eqref{euler} in the bounded $C^2$-domain $\Omega \subset \R^d$ for $d \geq 3$. Assume that $u$ is of H\"{o}lder regularity $C^{0,\gamma}(\Omega;\R^d)$ in the spatial variable for some $0<\gamma<1/2$. Then the hydrodynamic pressure $p$ is locally of double H\"{o}lder regularity $C^{0,2\gamma}_{\rm loc}(\Omega)$ in the spatial variable. 
Moreover, for any subdomain $\tilo\Subset\Omega$ and any $x_1, x_2 \in \tilo$,
\begin{align}\label{thm, est}
    |p(x_1)-p(x_2)|\leq C|x_1-x_2|^{2\gamma}
\end{align}
with the constant $
C=C\left(d,\gamma,\tilo\right)\cdot  \|u\|_{C^{0,\gamma}(\Omega)}^2$, where $C\left(d,\gamma,\tilo\right)$ depends only on $d$, $\gamma$, the $C^2$-geometry of $\Omega$, and the distance from $\tilo$ to $\p\Omega$.
\end{theorem}

\begin{remark}
In the above theorem, more precisely, one may replace the ``$C^2$-geometry of $\Omega$'' with the following geometrical data: the $C^0$-norm of the second fundamental form of $\po$, the intrinsic diameter of $\Omega$ (\emph{i.e.}, the infimum of the length of $C^1$-curves connecting any two points inside $\Omega$), and the injectivity radius of $
\Omega$. If a constant $C_{\rm geom}$ depends on these data only, we shall say that $C_{\rm geom}$ is a geometrical constant.
\end{remark}

\begin{remark}\label{remark: constant}
The constant $C$ in Theorem~\ref{thm}, \eqref{thm, est} may be chosen as follows. Setting $\kappa := \min\Big\{{\rm dist}(x_1, \p\Omega), {\rm dist}(x_2, \p\Omega)\Big\}>0$, we have that
\begin{equation*}
    C \approx \begin{cases}
1, \qquad \text{if } |x_1-x_2|\lesssim \kappa^{1+\frac{2\gamma}{d-2}}\text{ or } |x_1-x_2|\gtrsim 1;\\
\kappa^{\frac{-2\gamma(d-2+2\gamma)}{d-2}}, \qquad \text{if } \kappa^{1+\frac{2\gamma}{d-2}}\lesssim |x_1-x_2| \lesssim 1.   \end{cases}
\end{equation*}
All symbols $\approx$, $\lesssim$ here are understood as modulo constants depending only on $d$, $\gamma$, and the $C^2$-geometry of $\Omega$. 
\end{remark}

Throughout this note, we suppress the time variable by writing $u(t,x)\equiv u(x)$, $p(t,x)\equiv p(x)$, and the like. Theorem~\ref{thm} should be understood in the kinematic sense; that is, the inequality~\eqref{thm, est} holds for every $t \in [0,T]$.

Theorem~\ref{thm} was first established by Silvestre on the whole space $\R^d$ in the unpublished note \cite{note}. Such regularity results play an important role in the mathematical studies of turbulence theory, especially those pertaining to the Onsager conjecture. See \cite{turb1, turb2, turb3, turb4} and many of the subsequent works. Notably, the double H\"{o}lder regularity of the hydrodynamic pressure on $\R^d$ or $\mathbf{T}^d$ has been crucially used by Isett to prove the smoothness of (possibly nonunique) trajectories of Euler flows for velocities of below Lipschitz regularity \cite{i}, and recently by De Rosa--Isett to establish intermittency results in fully developed turbulence \cite{di}.

Bardos and Titi initiated in \cite{bt0} the project of extending the Onsager conjecture and related results to  bounded domains $\Omega \subset\R^d$. The H\"{o}lder regularity of the hydrodynamic pressure  $p$ in the Euler equations~\eqref{euler} plays a fundamental role in the study of anomalous dissipation on bounded domains \cite{bt0, btw, rrs1, rrs2}. Assuming $u \in C^{0,\gamma}(\Omega)$ with $0<\gamma<1$ and $\Omega$ is a bounded $C^3$-domain, Bardos and Titi obtained that $p \in C^{0,\gamma}_{\rm int}(\Omega)$ for $\Omega \subset \R^2$ in \cite{bt} and, together with Boutros, proved the same result for $\Omega \subset \R^3$ in \cite{bbt}.  Then, utilising among others the techniques developed in \cite{cd, cdf}, De Rosa, Latocca and Stefani \cite{dls1} established the almost double H\"{o}lder regularity in the sense that $p \in C^{1,\min(\alpha, 2\gamma-1)}(\Omega)$ for $1/2<\gamma<1$ on $C^{2,\alpha}$-domains, and $p \in C^{2\gamma-\e}(\Omega)$ with arbitrarily small $\e>0$ for $0<\gamma<1/2$ on $C^{3,\alpha}$-domains. The same authors later in \cite{dls2} extended the double H\"{o}lder regularity $p \in C^{2\gamma}(\Omega)$ to the index range $0<\gamma<1/2$ on $C^{2,1}$-domains. In addition, for the critical index $\gamma=1/2$, it is proved in \cite[Proposition~2.5]{dls2} that $p \in C^{1}_\star(\Omega)$, the Calder\'{o}n--Zygmund space. This improves the log-Lipschitz regularity of $p$ obtained earlier in 
Constantin \cite{c}. The interior double H\"{o}lder regularity of $p$ on $C^4$-domains in $\R^3$ has also been proved in \cite{bbt}.

The approach in the recent work \cite{dls2} by De Rosa, Latocca and Stefani involves a delicate application of  pseudodifferential calculus and the Littlewood--Paley theory. Here, on a bounded $C^2$-domain $\Omega \subset \R^{d\geq 3}$, we prove the interior double H\"{o}lder regularity $p \in C^{2\gamma}_{\rm loc}(\Omega)$ for $0<\gamma<1/2$, utilising only elementary techniques of (Neumann) Green functions and integration by parts. Moreover, we give explicit bounds on the blowup rate for $\frac{|p(x)-p(y)|}{|x-y|^{2\gamma}}$ as $x,y$ approach the boundary $\po$. The issue of double H\"{o}lder regularity of $p$ up to the boundary might have remained to be clarified, as indicated in \cite[p.8, at the end of \S~1.3]{bbt}.

Our arguments make essential use of ideas from Silvestre \cite{note} and Bardos--Boutros--Titi \cite{bbt}. As in \cite{bbt} (with an extra correction term), we introduce a \emph{modified pressure} $\wp$ (see \S\ref{sec: P} for details), which equals $$p+(u\cdot\nu)^2 + \dist(y,\po)\two\left(u^\top,u^\top\right)(y_\star)$$ on $\po$ and equals $p$ outside the $(2\delta)$-collar of the boundary, with  $\delta>0$ being a small parameter to be specified. Here, $\two$ is the second fundamental form of $\po$, $u^\top$ is the horizontal components of $u$, and $y_\star$ is the unique point on $\po$ such that $|y-y_\star| = \dist(y,\po)$. The construction of $\wp$ makes use of a cutoff function $\pd$ localised in the $(2\delta)$-collar of $\po$.

Rather than working directly with $\wp$, we introduce a \emph{second cutoff function} $\ed$ (see \eqref{second cutoff}) which equals zero in the $\delta/2$-collar of $\po$ and equals one outside the $\delta$-collar. Thus, by working with the elliptic PDE for $\ed\wp$ --- which is localised to the interior of $\Omega$ --- we circumvent the issues caused by  highly irregular boundary values (see Remark~\ref{remark: key} and \cite[Section~8]{bbt}). The cost to pay is that the singular terms $\left|\na\ed\right|\approx \delta^{-1}$ and $\left|D^2\ed\right|\approx \delta^{-2}$ enter our estimates. We overcome this issue by carefully adjusting the ``boundary layers'' induced by cutoff functions. For reasons that will become transparent along the proof, when estimating $|\wp(x_1)-\wp(x_2)|$ for $x_1$, $x_2$ sufficiently close to each other and  $|x_1-x_2| \approx \dist\left( \frac{x_1+x_2}{2},\po\right)$, we choose
\begin{align*}
\delta \approx |x_1-x_2|^{\frac{d-2+2\gamma}{d-2}}
\end{align*}
modulo geometrical constants.

With the above preparations, an adaptation of the representation formula for $p$ in Silvestre \cite{note} (with the integral kernel therein $\ppsi$ replaced by $\ed\ppsi$, where $\ppsi$ is the difference between the Neumann Green functions on $\Omega$ with singularities at $x_1$ and $x_2$) allows us to obtain the double H\"{o}lder regularity Theorem~\ref{thm} via direct potential estimates. It is crucial to our arguments that the previous choice of cutoff functions $\pd,\ed$ ensures that the constant in the inequality~\eqref{thm, est} is  uniform in $\delta$ up to the boundary. See Remark~\ref{remark: key} for details.

\subsection{Notations}
We fix some notations used throughout this paper.

We write $\1_E$ for the indicator function of a set $E$, ${\bf B}_\e(x)$ for the Euclidean ball centred at $x$ of radius $\e$, and  $\delta_y$ for the Dirac delta measure supported at $y$. For each $r>0$ sufficiently small (\emph{i.e.}, less than the injectivity radius of $\Omega$, which shall be recalled below), denote
\begin{equation*}
    \Omega_r := \left\{x \in \Omega:\,\dist(x,\po)>r \right\}.
\end{equation*}
Its complement $\Omega \setminus \Omega_r$ is the ``\emph{$r$-collar}'' of $\po$ in $\Omega$.

A continuous function $u:\tilo\to\R^\ell$ is \emph{$\alpha$-H\"{o}lder} ($0<\alpha<1$) if the following seminorm is finite:
\begin{align*}
    [u]_{C^{0,\alpha}\left(\tilo\right)} := \sup_{x\neq y\,\text{in } \tilo} \left\{ \frac{|u(x)-u(y)|}{|x-y|^\alpha} \right\}.
\end{align*}
Its \emph{$\alpha$-H\"{o}lder norm} is
\begin{align*}
    \|u\|_{C^{0,\alpha}\left(\tilo\right)} :=  [u]_{C^{0,\alpha}\left(\tilo\right)}  +  \|u\|_{C^{0}\left(\tilo\right)}. 
\end{align*}
A function $u:\Omega \to \R^\ell$ is in $C^{0,\alpha}_{\rm loc}(\Omega)$ if its restriction on $\tilo$ is in $C^{0,\alpha}\left(\tilo\right)$ for any $\tilo \Subset \Omega$.

Since $\Omega$ is a bounded $C^2$-domain, its boundary $\po$ is a compact $C^2$-hypersurface with the outward unit normal vector field $\nu \in C^1(\po,\R^d)$. Then the second fundamental form of $\po$, given by $$\two = \na\nu: \G\big(T(\po)\big)\times \G\big(T(\po)\big) \longrightarrow \R,$$ is of $C^0$-regularity; here $
\na$ is the Euclidean gradient on $\R^d$. The $C^2$-geometry of $\Omega$ is determined by the $C^0$-norm of the tensor $\two$. %A constant is said to be \emph{geometrical} if it depends only on the $C^2$-geometry of $\Omega$.

Starting from any $x\in\po$, one may flow the point $x$ by the inward unit normal $-\nu(x)$ for some time $\tau(x)$. The supremum of those numbers $\tau>0$ such that if $\tau(x) \leq \tau$ for all $x\in\po$ then the image of $\po$ under the flow has no self-intersections is the \emph{injectivity radius} of $\Omega$, denoted as ${\rm injrad}(\Omega)$. For the bounded $C^2$-domain $\Omega$, one has ${\rm injrad}(\Omega) > 0$. The distance function $y \mapsto \dist(y, \po)$ lies in $C^2\left(\Omega \setminus \overline{\Omega_\delta}\right)$ for any $\delta \in \left]0,{\rm injrad}(\Omega)\right[$.

Einstein's summation convention is assumed throughout. That is, repeated upper and lower indices are always understood as being summed over. Also, for a constant $C>0$, we write $C=C(a_1, a_2, \ldots, a_n)$ to emphasise that $C$ depends only on the parameters $a_1, a_2, \ldots, a_n$. If $a_1 = \Omega$ it means that $C$ depends on the $C^2$-geometry of $\Omega$ as well as $a_2, \ldots, a_n$.

\subsection{Organisation} 
In \S\ref{sec: P} we define the modified pressure $\wp$, which involves the first cutoff function $\pd$ and is a variant of the one introduced in Bardos, Boutros, and Titi \cite{bbt}. 

The main result, Theorem~\ref{thm}, will be proved in \S\ref{sec: proof}. The novel tool for our proof, namely that the second cutoff function $\ed$, is elaborated in \S\ref{subsec: ed}. Estimates for $\wp$ occupy \S\S\ref{subsec: x} $\&$ \ref{subsec: y}. Finally, the proof of Theorem~\ref{thm} is concluded in \S\ref{subsec: z}.

Several concluding remarks of the paper are given in \S\ref{sec: concl}.

\section{The modified pressure}
\label{sec: P}
In this section, motivated by Bardos, Boutros and Titi \cite[(2.17)]{bbt},  we introduce the modified pressure $\wp$ from the problem \eqref{euler}.

Let $\phi^\delta:[0,
\infty[ \to [0,1]$ be a smooth non-increasing function such that
\begin{equation*}
\phi^\delta(s) \equiv 1 \text{ for $0\leq s \leq \delta$}, \quad \phi^\delta(s) \equiv 0 \text{ for $s \geq 2\delta$},\quad \text{ and } \left|\left(\phi^\delta\right)'\right|\leq \frac{2}{\delta},
\end{equation*}
for an arbitrary small $\delta>0$. To fix the idea,   assume throughout 
\begin{align*}
    0<\delta<10^{-3}.
\end{align*}
Then set 
\begin{equation}\label{first cutoff}
\pd(y) := \phi^\delta\big(\dist(y,\po)\big)\qquad \text{for } y \in \Omega.
\end{equation}
It is supported in the $(2\delta)$-collar of the boundary, namely that $\Omega \setminus \Omega_{2\delta}$, and is termed as the \emph{first cutoff function}. As $\Omega$ is a $C^2$-domain, we have $\pd \in C^2(\Omega)$.  

Let $\nor$ be an extension in $\Omega \setminus \Omega_{10\delta}$ of the outward unit normal vector field $\nu$  on $\po$, in the way that for any $y \in \Omega \setminus \Omega_{10\delta}$, there is a unique point --- \emph{i.e.}, the nearest point projection --- $y_\star \in \po$ such that $|y-y_\star|=\dist(y,\po)$, we set $\nor(y):=\nu(y_\star)$, by noting (for reasons that shall become clear from the later developments) that 
\begin{equation}\label{condition on injrad}
{\rm injrad}(\Omega) \geq 100 \cdot \delta^{\frac{d-2+2\gamma}{d-2}}.    
\end{equation}

As in Bardos, Boutros, and Titi  \cite[(2.17)]{bbt}, we set 
\begin{equation}\label{modified pressure}
    P \equiv P^\delta := p + \pd \left(u\cdot\nor\right)^2,
\end{equation}
where $p$ is the hydrodynamic pressure in the Euler equation~\eqref{euler}. For ease of notations, we suppress its dependence on the cutoff parameter $\delta$. This is our ``first modified pressure''. We also denote
\begin{equation}\label{new-Q, def}
    \QQ := \pd \left(u\cdot\nor\right)^2.
\end{equation}
Note that ${\bf supp}\,\QQ \subset \Omega \setminus \Omega_{2\delta}$ and $\QQ \equiv \left(u\cdot\nor\right)^2$ in $ \Omega \setminus \Omega_\delta$. 

\begin{remark}\label{remark: key}
It is crucial to notice that $\QQ$ does not have a well-defined normal derivative on the boundary, cf. \cite[Section~8]{bbt}. A divergence-free $C^{0,\gamma}$-vector field $u$ is constructed therein such that 
$$u\cdot\na\left( u\cdot\nu\right) \big|_{\po}\notin \mathcal{D}'(\po).$$
\end{remark}

Nevertheless, although $\p_\nu\QQ\big|_{\po}$ and $\p_\nu p\big|_{\po}$ fail to be well-defined by themselves in general, $\p_\nu P\big|_{\po}$ is a good quantity. This is because
\begin{align*}
    \p_\nu p = \na\nor: (u\otimes u) - \p_\nu\left(u\cdot\nor\right)^2 - \p_\tau[(u\cdot \tau)(u\cdot\nor)] - \p_t(u\cdot\nor),
\end{align*}
where $\tau$ denotes vector fields tangent to $\po$, and $u\cdot\nor\big|_{\po} = u\cdot\nu\big|_{\po} = 0$. See  \cite[pp.2--3]{bbt}. This indeed is the motivation for the introduction of $P$ in \cite{bbt}. 

With the notations above, we deduce from the Euler equations~\eqref{euler} the following Neumann boundary value problem for our first modified pressure:
\begin{equation}\label{key PDE for P}
    \begin{cases}
\Delta P = \Delta \QQ + \dduu \qquad \text{in } [0,T]\times \Omega,\\
\p_\nu P = \two\left(u^\top,u^\top\right)\qquad \text{on } [0,T]\times \po.
    \end{cases}
\end{equation}
Here $\two$ is the second fundamental form of the surface $\po \subset\R^3$ which maps a pair of tangential vector fields along $\po$ to a scalar. One has that $$\two\left(u^\top,u^\top\right) \equiv \na \nu : (u\otimes u),$$ where the superscript ${}^\top$ denotes the  projection of a vector field to the tangential plane of $\po$.

To proceed, let us further introduce a \emph{second modified pressure} $\wp$, which satisfies the homogeneous Neumann boundary condition and hence admits a representation formula via the Neumann Green function (see \eqref{ppsi def} and the ensuing lines). For this purpose, we set
\begin{equation}\label{new, correction R}
\RR(y) := \dist(y,\po) \pd(y)\left[\two\left(u^\top, u^\top\right)(y_\star) \right],
\end{equation}
where $y_\star$ is the nearest point projection of $y$ to $\po$ as before. It follows that
\begin{align*}
\p_{\nor}\RR(y) &= \left\{\p_{\nor}\dist(y,\po) \pd(y) + \dist(y,\po)\p_{\nor} \pd(y) \right\}
\left[\two\left(u^\top, u^\top\right)(y_\star) \right]\\
&=-\left[\two\left(u^\top, u^\top\right)(y_\star) \right]\qquad \text{for any } y \in \Omega \setminus \Omega_{\delta},
\end{align*}
thanks to $\pd\equiv 1$ and $\p_{\nor}\dist(y,\po)=-1$ in $\Omega \setminus \Omega_{\delta}$. Thus,  setting
\begin{equation}\label{wp, def}
    \wp := P+\RR \equiv p + \QQ + \RR
\end{equation}
for $\QQ$ and $\RR$ in \eqref{new-Q, def} and \eqref{new, correction R}, respectively, we obtain the \emph{homogeneous} Neumann boundary problem for $\wp$:
\begin{equation}\label{key PDE for wp}
    \begin{cases}
\Delta \wp = \Delta \QQ + \Delta\RR + \dduu \qquad \text{in } [0,T]\times \Omega,\\
\p_\nu \wp = 0 \qquad \text{on } [0,T]\times \po.
    \end{cases}
\end{equation}

Equation~\eqref{key PDE for wp} serves as the starting point of all our subsequent developments. It is straightforward to check that \eqref{key PDE for P} and \eqref{key PDE for wp} simultaneously satisfy the compatibility condition for the Neumann problem of the Poisson equation. For simplicity, we refer to $\wp$ as the \emph{modified pressure} from now on. Without loss of generality, we assume the normalisation condition: $$\int_\Omega \wp(y)\,\dd y = 0.$$ Indeed, the solutions $\wp$ may differ from the one with average zero by addition of some function whose gradient is a harmonic vector field, which is nontrivial on non-simply-connected $\Omega$. But such a function is in  $C^\infty(\Omega)$ when $\Omega$ is a $C^2$-domain, hence does not affect the double H\"{o}lder regularity of $\wp$.

\section{Proof of Theorem~\ref{thm}}
\label{sec: proof}

In this section, we prove our main Theorem~\ref{thm} by establishing the bound for $\wp$ in the $C^{0,2\gamma}(\Omega)$-norm from the homogeneous Neumann problem~\eqref{key PDE for wp}. The key new idea is to introduce, other than $\pd$ in \eqref{wp, def}, a second cutoff function localised in $\overline{\Omega_{\delta/2}\setminus \Omega_\delta}$. 

\subsection{The second cutoff function}\label{subsec: ed}

We begin with controlling $\left|\wp(x_1)-\wp(x_2)\right|$ for $x_1, x_2 \in\Omega$ satisfying $|x_1-x_2|\ll 1$. Let us consider any $x_1$ and $x_2$ in $\Omega$ satisfying that
\begin{equation}\label{new, choice of x1, x2} 
\begin{cases}
\text{the segment } [x_1, x_2] \subset \Omega_{10 \cdot \delta^\beta}  \text{ and } 
  c_\Omega\delta\leq  |x_1-x_2| \leq c_\Omega\delta^\beta,\\
\text{where  $\beta := \frac{d-2}{d-2+2\gamma} \in ]0,1[$ and $c_\Omega > 0$ is a geometrical constant.}
\end{cases}
\end{equation}
See Theorem~\ref{thm} for the notion of geometrical constants. Such a choice of $x_1$ and $x_2$ is possible because $\Omega$, as a bounded $C^2$-domain, satisfies the uniform interior sphere condition. That is, there exists a uniform constant $r_0>0$ depending only on the $C^2$-geometry of $\Omega$ such that for each $x \in \po$, the open ball  of radius $r_0$ tangent to $\po$ at $x$ lies entirely in $\Omega$. %Call any such $\mathcal{B}$ a ``boundary $r_0$-ball rooted at $x$'' and denote $\mathcal{B}\equiv \mathcal{B}[r_0;x]$. Now, for any $x_1\in \Omega$ lying in the subset $\mathcal{B}[r_0/10; x_\star]$ of some boundary $r_0$-ball $\mathcal{B}[r_0; x_\star] \subset \Omega$, take $\e>0$ so small (depending on $r_0$) that $|x_1-x_\star|>30\e$ and  $\dist(x_1,\po)>20\e$. Then let $c_\Omega >0$ be a geometrical constant such that the Euclidean ball ${\bf B}_{2c_\Omega\cdot\e}(x_1) \subset \mathcal{B}[r_0/5; x_\star]$. Thus, by triangle inequality, for any $x \in {\bf B}_{2c_\Omega\cdot\e}(x_1)$ we have that $\dist(x,\po)>10\e$ ({\color{red} ???}). The same results clearly hold for $x_1$ lying outside $\mathcal{B}[r_0/10; x_\star]$. (Here, as before, $\mathcal{B}[r_0; x_\star]$ is a boundary $r_0$-ball.) Taking $\e=\delta^\beta$, we thus conclude that $[x_1, x_2] \subset \Omega_{10 \cdot \delta^\beta}$ and $|x_1-x_2| \leq c_\Omega\delta^\beta$. ({\color{red} what is $x_2$?})

The reason for choosing $\beta:=\frac{d-2}{d-2+2\gamma}$ will become transparent later (see the formula \eqref{3.8} in the proof of Lemma~\ref{lemma: I1}). The further restriction $|x_1-x_2| \geq c_\Omega\delta$ in \eqref{new, choice of x1, x2} will not be invoked until Step~4 of the proof of Lemma~\ref{lemma: I2}. %This restriction is only temporary and shall be removed in the first step of the proof of Theorem~\ref{thm} in \S\ref{subsec: z}.

%Notice too that we have the freedom of choosing $|x_1-x_2|$ (hence $\delta$) arbitrarily small, since our goal is to bound the H\"{o}lder (semi)norm of the modified pressure $\wp$. 

To resume, let $\ppsi$ be the distributional solution to the following Neumann problem, subject to $\int_\Omega\ppsi\,\dd y=0$:
\begin{equation}\label{ppsi def}
    \begin{cases}
        \Delta \ppsi = \delta_{x_1} - \delta_{x_2} \qquad \text{ in } \Omega,\\
        \p_\nu\ppsi = 0 \qquad\text{ on } \po.
    \end{cases}
\end{equation}
Equivalently, we take 
\begin{equation*}
    \ppsi(x) = \kn(x-x_1) - \kn(x-x_2)\qquad\text{for } x \in \Omega,
\end{equation*}
where $\kn$ is the \emph{Neumann Green function} on $\Omega$. It is known that
\begin{align}\label{kernel estimate}
    \left|D^\ell \kn(z)\right| \lesssim |z|^{-d+2-\ell}\qquad\text{for each } \ell \in\{0,1,2\}.
\end{align}
modulo a uniform constant depending only on the dimension $d$ and the geometry of $\Omega$.

For the claim above, when $\Omega$ is a bounded $C^{2,\alpha}$-domain in $\R^d$ for any $\alpha>0$, see \cite[Appendix~C]{dls1} for a detailed proof. Now let $\Omega$ be a bounded $C^{2}$-domain. By  standard boundary straightening  via $C^{2}$-diffeomorphisms in local charts covering the boundary $\po$, as well as an even extension of the solution across the boundary after straightening, the question is transformed to the pointwise estimates for the \emph{Dirichlet} Green function up to the second derivatives for uniformly elliptic operators with $C^2$-coefficients. In this case, the estimates of the form~\eqref{kernel estimate} for the Dirichlet Green function hold on $C^2$-domains; \emph{e.g.}, by adapting the arguments for \cite[p.120, Theorem~6.25]{gtbook}.

Observe also that $\ppsi$ is regular away from $x_1$ and $x_2$; more precisely, it is $C^2$ in the $(2\delta)$-collar of the boundary, \emph{i.e.}, $\Omega \setminus \Omega_{2\delta}$.

A natural adaptation of the arguments in \cite{note} would be testing the equation~\eqref{key PDE for wp} against $\ppsi$. However, integration by parts formally yields that $\int_\Omega (\Delta\QQ)\ppsi\,\dd x = \int_{\po} \ppsi (\p_\nu\QQ)\,\dd\Sigma$, while the right-hand side may be undefined in view of Remark~\ref{remark: key}!

The key novel ingredient of this note is the introduction of a second cut-off function, which effectively circumvents the above issue. Let $$\varpi^\delta: [0,\infty[ \longrightarrow [0,1]$$ be a smooth non-decreasing function such that
\begin{equation*}
\varpi^\delta(s) \equiv 0 \text{ for $0\leq s \leq \frac{\delta}{2}$}, \quad \varpi^\delta(s) \equiv 1 \text{ for $s \geq \delta$},\quad \text{ and } \left(\varpi^\delta\right)'\leq \frac{4}{\delta},
\end{equation*}
for  an arbitrary positive number $\delta$. Then set 
\begin{equation}\label{second cutoff}
\ed(y) := \varpi^\delta\big(\dist(y,\po)\big)\qquad \text{for } y \in \Omega.
\end{equation} 
As $\Omega$ is a $C^2$-domain, we have $\ed \in C^2(\Omega)$.

Noticing that  ${\bf supp}\,(\ed) \subset \overline{\Omega_{\delta/2}}$, we obtain via integration by parts that
\begin{align*}
&\int_\Omega (\Delta\QQ)(y)\left(\ed\ppsi\right)(y)\,\dd y \\
&\qquad= \int_\Omega \QQ(y) \Delta \left(\ed\ppsi\right)(y)\,\dd y\\
&\qquad= \int_\Omega \QQ(y)\left\{ \ppsi(\Delta\ed) +2\na\ppsi \cdot \na\ed + \ed\Delta\ppsi \right\}(y)\,\dd y.
\end{align*}
Here $\QQ\eta^\delta$ in continuous (indeed, $C^{0,\gamma}$) and $\ed \equiv 1$ on $\Omega_\delta \supset \Omega_{10\cdot\delta^\beta} \supset [x_1,x_2]$ (recall the condition~\eqref{new, choice of x1, x2}), so
\begin{align*}
\int_\Omega \QQ\ed\Delta\ppsi \,\dd y = \left(\QQ\ed\right)(x_1) - \left(\QQ\ed\right)(x_2) = \QQ(x_1)-\QQ(x_2), 
\end{align*}
where we have used the  definition of $\ppsi$ given in \eqref{ppsi def}. But by construction  ${\bf supp}\,(\QQ) \subset \Omega \setminus \Omega_{2\delta}$, disjoint from $\Omega_{10\cdot\delta^\beta}$, which implies that $$\QQ(x_1)=\QQ(x_2)=0.$$

Thus, testing the equation~\eqref{key PDE for wp} against $\left(\ed\ppsi\right)$, one obtains
\begin{align*}
\int_\Omega \left(\Delta\pbad\right)\left(\ed\ppsi\right)\,\dd y &= \int_\Omega (\QQ+\RR)\left\{\ppsi \Delta\ed +2 \na\ppsi \cdot\na\ed \right\}\,\dd y\nonumber\\
&\qquad + \int_\Omega \left[\dduu\right]\left(\ed\ppsi\right)\,\dd y.
\end{align*}
As $u$ is divergence-free, we have 
\begin{align*}
    \left[\dduu\right](y) = {\rm div}\,{\rm div}\,\left[ \big(u(y)-u(x_1)\big)\big(u(y)-u(x_2)\big) \right]. 
\end{align*}
See Silvestre \cite{note}. A further integration by parts applied to the right-most term yields that 
\begin{align*}
&\int_\Omega \left(\Delta\wp\right)\left(\ed\ppsi\right)\,\dd y \nonumber\\
&\qquad = \int_\Omega (\QQ+\RR)\left\{\ppsi \Delta\ed +2 \na\ppsi \cdot\na\ed \right\}\,\dd y\nonumber\\
&\qquad\qquad + \int_\Omega \big(u^i(y)-u^i(x_1)\big)\big(u^j(y)-u^j(x_2)\big) \p_{i}\p_j \left(\ed\ppsi\right)(y)\,\dd y\nonumber\\
&\qquad =\int_{\Omega_{\delta/2}\setminus\Omega_\delta} \left[(u\cdot\nor)^2 + \dist(y,\po)\two\left(u^\top, u^\top\right)(y_\star)\right]\cdot\nonumber\\
&\qquad\qquad\qquad\qquad\cdot\left\{\ppsi \Delta\ed +2 \na\ppsi \cdot\na\ed \right\}\,\dd y\nonumber\\
&\qquad\qquad + \int_\Omega \big(u^i(y)-u^i(x_1)\big)\big(u^j(y)-u^j(x_2)\big) \p_{i}\p_j \left(\ed\ppsi\right)(y)\,\dd y,
\end{align*}
by noting that the derivatives of $\ed$ are supported in the annulus $\overline{\Omega_{\delta/2}\setminus\Omega_\delta}$, and that $\pd\equiv 1$ thereof. 
The Einstein summation convention is adopted here. 
Finally, for the left-most term, we proceed as for the term involving $(\QQ+\RR)$ to deduce that
\begin{align*}
&\int_\Omega \left(\Delta\wp\right)\left(\ed\ppsi\right)\,\dd y \\
&\qquad = \pbad(x_1) - \pbad(x_2) + \int_\Omega \pbad\left\{\ppsi \Delta\ed +2 \na\ppsi \cdot\na\ed \right\}\,\dd y.
\end{align*}
Again, the presence of the second cutoff function $\ed$ ensures that no boundary term arises from the integration by parts.

Summarising the above two identities, we obtain that
\begin{align}\label{pbad key eq}
&\pbad(x_1) - \pbad(x_2) \nonumber\\
&\qquad=\int_{\Omega_{\delta/2}\setminus\Omega_\delta} \left[(u\cdot\nor)^2+ \dist(y,\po)\two\left(u^\top, u^\top\right)(y_\star)-\pbad\right] \cdot\nonumber\\
&\qquad\qquad\qquad\qquad \cdot\left\{\ppsi \Delta\ed +2 \na\ppsi \cdot\na\ed \right\}\,\dd y\nonumber\\
&\qquad\qquad + \int_\Omega \big(u^i(y)-u^i(x_1)\big)\big(u^j(y)-u^j(x_2)\big) \p_{i}\p_j \left(\ed\ppsi\right)(y)\,\dd y \nonumber\\
&\qquad =: I_1+I_2.
\end{align}
This is the starting point of our analysis below. %Notice here that $(u\cdot\nor)^2-\pbad = \wp-\pgood$ 

\subsection{Estimate for $I_1$}\label{subsec: x}
Let us first bound $I_1$, the integral term over the annulus $\Omega_{\delta/2}\setminus\Omega_\delta$. We introduce the shorthand notation:
\begin{align}\label{new-M}
    M:=(u\cdot\nor)^2+ \dist(y,\po)\two\left(u^\top, u^\top\right)(y_\star)-\pbad
    \qquad\text{on $\Omega_{\delta/2}\setminus\Omega_\delta$.}
\end{align}
It has been  established in  \cite{dls1} that $M \in C^{0,\gamma}\left(\overline{\Omega_{\delta/2}\setminus\Omega_\delta}\right)$. In fact, this can be seen from the simple argument below: 
\begin{itemize}
    \item 
$(u\cdot\nor)^2 \in C^{0,\gamma}(\Omega)$ by the assumption on $u$;
\item 
$\dist(y,\po)\two\left(u^\top, u^\top\right)(y_\star)\in C^{0,\gamma}(\Omega)$ since $\two \in C^0(\po)$ on the $C^2$-domain $\Omega$, and the distance function $\dist(\bullet,\po)$ as well as the nearest point projection $y \mapsto y_\star$ are both $C^2$ in $\Omega_{\delta/2}\setminus\Omega_\delta$;
\item 
 $\pbad$ satisfies the Neumann problem~\eqref{key PDE for wp} of the form $\Delta \pbad = D^2 F$, where $F\in C^{0,\gamma}(\Omega)$ is quadratic in $u$. Thus, \emph{away from the boundary $\po$}, we have $\pbad \in C^{0,\gamma}(\Omega_{\delta/2})$.
\end{itemize}

In view of \eqref{wp, def} and the above argument, we also have $\QQ,\RR, u\otimes u \in C^{0,\gamma}(\Omega)$. Solving the standard Neumann problem from \eqref{key PDE for wp} and using the definition of $\QQ$ and $\RR$, we have $\wp \in C^{0,\gamma}(\Omega)$ and hence 
\begin{equation}\label{p, C0, June25}
\|p\|_{C^{0,\gamma}(\Omega)} \leq C\left(d,\gamma, \Omega\right)\|u\|_{C^{0,\gamma}(\Omega)}^2.
\end{equation}

In fact, we shall only use $M \in C^{0}\left(\overline{\Omega_{\delta/2}\setminus\Omega_\delta}\right)$ and $p \in C^0(\Omega)$ in the sequel. 

\begin{lemma}\label{lemma: I1}
Let $x_1, x_2 \in \Omega$ be as in condition~\eqref{new, choice of x1, x2}, and let $M$ be as in \eqref{new-M}.  The term \begin{align*}
    I_1 = \int_{\Omega_{\delta/2}\setminus\Omega_\delta} M(y)\left\{\ppsi \Delta\ed +2 \na\ppsi \cdot\na\ed \right\}(y)\,\dd y
\end{align*}
satisfies the H\"{o}lder estimate: \begin{equation*}
    |I_1| \leq C|x_1-x_2|^{2\gamma},
\end{equation*}
where the constant $C$ depends only on the dimension $d$ and $\|M\|_{C^{0}\left(\overline{\Omega_{\delta/2}\setminus\Omega_\delta}\right)}$.
\end{lemma}

\begin{proof}[Proof of Lemma~\ref{lemma: I1}]
We split $I_1=I_{1,1}+I_{1,2}$, where
\begin{align*}
&I_{1,1} := \int_{\Omega_{\delta/2}\setminus\Omega_\delta} M \ppsi \Delta\ed\,\dd y,\\
&I_{1,2} :=2\int_{\Omega_{\delta/2}\setminus\Omega_\delta} M \na\ppsi \cdot\na\ed\,\dd y.
\end{align*}

For $I_{1,1}$, we have the following pointwise estimates for the kernel function $\ppsi(y)\equiv\kn(y-x_1)-\kn(y-x_2)$ with $y \in \Omega_{\delta/2}\setminus\Omega_\delta$:
\begin{align*}
\left|\ppsi(y)\right| \leq C_1 \left\{|y-x_1|^{2-d} + |y-x_2|^{2-d} \right\} \leq C_2 \delta^{\beta(2-d)}
\end{align*}
and
\begin{align*}
   \left|\ppsi(y)\right| \leq \left(\sup_{\xi \in [x_1,x_2]}\left|\na \kn(y-\xi)\right|\right) |x_1-x_2| \leq C_3 \delta^{\beta(1-d)}|x_1-x_2|,
\end{align*}
where $C_i$ are dimensional constants and $x_1, x_2 \in \Omega$ are as in condition~\eqref{new, choice of x1, x2}. In the above we used the Taylor expansion, the inequalities $|y-\xi| \geq |\xi|-|y| \geq 10\delta^\beta - \delta/2 \geq 9\delta^\beta$, and the estimate~\eqref{kernel estimate} for the Neumann Green function.

Combining the two estimates above, we may bound
\begin{align*}
      \left|\ppsi(y)\right| &\leq C_4 \left\{\delta^{\beta(2-d)}\right\}^{1-2\gamma} \cdot \left\{\delta^{\beta(1-d)}|x_1-x_2|\right\}^{2\gamma}\\
      &= C_4 \delta^{\beta(2-d-2\gamma)}\cdot |x_1-x_2|^{2\gamma}
\end{align*}
for some $C_4=C_4(d)$. Hence, noting that $\left|\Delta\ed(y)\right|\lesssim {\delta^{-2}}$, we obtain 
\begin{align*}
    |I_{1,1}| &\leq C_5\, {\rm Volume}\left(\Omega_{\delta/2}\setminus\Omega_\delta\right) \cdot  \delta^{-2} \cdot  \delta^{\beta(2-d-2\gamma)}\cdot |x_1-x_2|^{2\gamma}\\
    &\leq C_6 \,\delta^{\beta(2-d-2\gamma) +d-2}\cdot |x_1-x_2|^{2\gamma},
\end{align*}
where $C_5$ and $C_6$ depend only on $d$ and $\|M\|_{C^{0}\left(\overline{\Omega_{\delta/2}\setminus\Omega_\delta}\right)}$. Substituting in the choice of parameter 
\begin{equation}\label{3.8}\beta := \frac{d-2}{d-2+2\gamma},
\end{equation}
we arrive at 
\begin{align}\label{I11}
    |I_{1,1}| \leq C_6|x_1-x_2|^{2\gamma}.
\end{align}
In particular, $C_6$ is uniform in $\delta$.

Next let us estimate $I_{1,2}$. For $y \in \Omega_{\delta/2}\setminus\Omega_\delta$, it holds that
\begin{align*}
\left|\na \ppsi(y)\right| \leq C_7 \left\{|y-x_1|^{1-d} + |y-x_2|^{1-d} \right\} \leq C_8 \delta^{\beta(1-d)},
\end{align*}
as well as that
\begin{align*}
   \left|\na \ppsi(y)\right| \leq \left(\sup_{\xi \in [x_1,x_2]}\left|D^2 \kn(y-\xi)\right|\right) |x_1-x_2| \leq C_9 \delta^{-\beta d}|x_1-x_2|
\end{align*}
for $x_1, x_2 \in \Omega$ as in condition~\eqref{new, choice of x1, x2}. As before, we used here the Taylor expansion, the inequalities $|y-\xi| \geq  9\delta^\beta$, and the estimate~\eqref{kernel estimate}. Thus
\begin{align*}
\left|\na \ppsi(y)\right| &\leq C_{10} \left\{\delta^{\beta(1-d)}\right\}^{1-2\gamma} \cdot \left\{\delta^{-\beta d}|x_1-x_2|\right\}^{2\gamma}\\
&= C_{10} \delta^{\beta(1-d-2\gamma)} \cdot |x_1-x_2|^{2\gamma}.
\end{align*}
The above constants $C_k$ ($7\le k\le 10$) are dimensional. In view of the choice of $\beta$ given in \eqref{3.8} and that $\left|\na\ed(y)\right|\lesssim {\delta^{-1}}$, we thus have
\begin{align}\label{I12}
    |I_{1,2}| &\leq C_{11} {\rm Volume}\left(\Omega_{\delta/2}\setminus\Omega_\delta\right) \cdot  \delta^{-1} \cdot  \delta^{\beta(1-d-2\gamma)}\cdot |x_1-x_2|^{2\gamma}\nonumber\\
    &\leq C_{12} \delta^{\beta(1-d-2\gamma) +d-1}\cdot |x_1-x_2|^{2\gamma}\nonumber\\
    &= C_{12}\delta^{d-1 - \frac{d-2}{d-2+2\gamma}\cdot(d-1+2\gamma)} \cdot |x_1-x_2|^{2\gamma}.
\end{align}
Here $C_{11}$ and $C_{12}$ depend on $d$ and $\|M\|_{C^{0}\left(\overline{\Omega_{\delta/2}\setminus\Omega_\delta}\right)}$ only. The index $d-1 - \frac{d-2}{d-2+2\gamma}\cdot(d-1+2\gamma)$ on the right-most term of \eqref{I12} is strictly positive. 
The assertion in Lemma~\ref{lemma: I1} now follows from the inequalities~\eqref{I11} and \eqref{I12}.  
\end{proof}

 \subsection{Estimate for $I_2$}\label{subsec: y}
For this purpose, we establish the following
\begin{lemma}\label{lemma: I2}
Let $x_1, x_2 \in \Omega$ be as in condition~\eqref{new, choice of x1, x2}.   The term \begin{align}\label{I2, expression}
    I_2:=\int_\Omega \big(u^i(y)-u^i(x_1)\big)\big(u^j(y)-u^j(x_2)\big) \p_{i}\p_j \left(\ed\ppsi\right)(y)\,\dd y
\end{align}
satisfies the H\"{o}lder estimate: \begin{equation*}
    |I_2| \leq C|x_1-x_2|^{2\gamma},
\end{equation*}
where  $C$ depends only on $d$, $\gamma$, $\|u\|_{C^{0,\gamma}(\Omega)}$, and the $C^2$-geometry of $\Omega$.
\end{lemma}

This lemma follows essentially from Silvestre \cite[Section~1.2]{note}. Some new estimates are needed to deal with the technicalities brought about by the second cutoff function $\ed$. %To make our paper self-contained, we reproduce the arguments here.

\begin{proof}[Proof of Lemma~\ref{lemma: I2}]

We divide our arguments into five steps below.

\smallskip
\noindent
{\bf 1.} We first remark that the expression~\eqref{I2, expression} for $I_2$ makes sense, despite that $\p_{i}\p_j \left(\ed\ppsi\right)$ has singularities at $y=x_1$ and $x_2$. Indeed, one may write the integral in the principal value meaning: $${\rm p.v.}\int_
\Omega \left\{\cdots\right\}\,\dd y = \lim_{\e\searrow 0} \int_{\Omega \setminus ( {\bf B}_\e(x_1) \cap {\bf B}_\e(x_2))}\left\{\cdots\right\}\,\dd y.$$  Near $x_1$ we have that $\big|u^i(y)-u^i(x_1)\big|\lesssim |y-x_1|^{\gamma}$ and $\left| \p_{i}\p_j \left(\ed\ppsi\right)(y)\right|\lesssim |y-x_1|^{-d}$, thanks to \eqref{kernel estimate}. Thus, their product is controlled by $|y-x_1|^{-d+\gamma}$, which is locally integrable in ${\bf B}_\e(x_1)$ and hence vanishes in the limit $\e\searrow 0$. The argument near $x_2$ is completely parallel.

\smallskip
\noindent
{\bf 2.}  Denote
\begin{align*}
    \Bar{x}:=\frac{x_1+x_2}{2}\qquad\text{and}\qquad \rho = |x_1-x_2|,
\end{align*}
and also set
\begin{align*}
    {\Omega_{\rm far}}:= \left\{y \in \Omega:\, |y-\bar{x}|> 5\rho\right\} \qquad\text{and}\qquad  {\Omega_{\rm near}}:= \left\{y \in \Omega:\, |y-\bar{x}|\leq 5\rho\right\} .
\end{align*}
We split $I_2$ into
\begin{align*}
I_2 &= I_{2,1}+I_{2,2}\\
&:= \left\{ \int_{\Omega_{\rm near}} + \int_{\Omega_{\rm far}} \right\} \big(u^i(y)-u^i(x_1)\big)\big(u^j(y)-u^j(x_2)\big) \p_{i}\p_j \left(\ed\ppsi\right)(y)\,\dd y.
\end{align*}

\smallskip
\noindent
{\bf 3.} For $I_{2,1}$, we proceed in a way similar to that given in \cite[the end of p.2]{note}.

For any $y \in \Omega_{\rm near}$, we have $|y-\bar{x}| \leq 5\rho \leq 5c_\Omega\delta^\beta$. (Recall from condition~\eqref{new, choice of x1, x2} that $\rho=|x_1-x_2|\leq c_\Omega\delta^\beta$ for $\beta=\frac{d-2}{d-2+2\gamma}$, and that the segment $[x_1, x_2] \subset \Omega_{10 \cdot \delta^\beta}$.) The triangle inequality yields that so $y \in \overline{\Omega_{5\delta^\beta}}$. In particular, $y$ is not in the annulus $\overline{\Omega_{\delta/2}\setminus\Omega_\delta}$, which contains the support of derivatives of $\ed$ (see \eqref{second cutoff}). We thus have the identity
\begin{align*}
 D^2\group \equiv \ed D^2\ppsi\qquad\text{on } \Omega_{\rm near},
\end{align*}
which implies that
\begin{align*}
    \left|D^2\group (y)\right| \1_{\Omega_{\rm near}}(y) &\leq \left|D^2\ppsi (y)\right|\\
    &\leq \left|D^2\kn (y-x_1)\right|+\left|D^2\kn (y-x_2)\right|.
\end{align*}

So, one has that
\begin{align*}
\left|I_{2,1}\right|&\leq \int_{\Omega_{\rm near}}\big|u^i(y)-u^i(x_1)\big|\big|u^j(y)-u^j(x_2)\big|\\
&\qquad\qquad\qquad\qquad \times\Big\{\left|\p_i\p_j\kn (y-x_1)\right|+\left|\p_i\p_j\kn (y-x_2)\right|\Big\}\,\dd y\\
&\leq C\|u\|_{C^{0,\gamma}(\Omega)} \rho^\gamma\Bigg\{\int_{\Omega_{\rm near}}\left|u(y)-u(x_1)\right|\left|D^2\kn(y-x_1)\right|\,\dd y\\
&\qquad\qquad\qquad\qquad  + \int_{\Omega_{\rm near}}\left|u(y)-u(x_2)\right|\left|D^2\kn(y-x_2)\right| \Bigg\}\,\dd y,
\end{align*}
where $C$ is a universal constant. Furthermore, we have
\begin{align*}
&\int_{\Omega_{\rm near}}\left|u(y)-u(x_1)\right|\left|D^2\kn(y-x_1)\right|\,\dd y \\
&\qquad\qquad\leq C [u]_{C^{0,\gamma}(\Omega)}\int_{\Omega_{\rm near}} \frac{1}{|y-x_1|^{d-\gamma}}\,\dd y
\end{align*}
for some $C=C(d)$. Thanks to the triangle inequality, it holds that 
\begin{align*}
\int_{\Omega_{\rm near}} \frac{1}{|y-x_1|^{d-\gamma}}\,\dd y &\leq \int_{{\bf B}_{5.5\rho}(x_1)}\frac{1}{|y-x_1|^{d-\gamma}}\,\dd y \\
&= C\rho^\gamma
\end{align*}
for another dimensional constant $C=C(d)$. The estimate for the term $\int_{\Omega_{\rm near}}\left|u(y)-u(x_2)\right|\left|D^2\kn(y-x_2)\right|$ is completely parallel.

Therefore, collecting the estimates above, we arrive at
\begin{equation}\label{I21 estimate}
    |I_{2,1}| \leq C|x_1-x_2|^{2\gamma},
\end{equation}
where $C$ depends on $d$ and the H\"{o}lder norm $\|u\|_{C^{0,\gamma}(\Omega)}$. 

\smallskip
\noindent
{\bf 4.} Finally we turn to $I_{2,2}$. This is the most technical term, since on $\Omega_{\rm far}$ the derivatives of the second cutoff function $\ed$ are not everywhere vanishing. Our treatment is reminiscent of the proof of Lemma~\ref{lemma: I1}.

For this purpose, we write
\begin{align*}
    D^2\group = \left(D^2 \ed\right) \ppsi + 2\na\ed\otimes\na\ppsi + \ed D^2\ppsi\qquad\text{on } \Omega_{\rm far}.
\end{align*}
Utilising the Taylor expansion, triangle inequality, and the estimate~\eqref{kernel estimate} for the Neumann Green function, we deduce that 
\begin{align*}
\left|D^\ell\ppsi(y)\right| \leq C\frac{\rho}{|y-\bar{x}|^{d+\ell-1}} \qquad\text{for } \ell \in \{0,1,2\},
\end{align*}
where $C=C(d,\ell)$. This together with $\left|D^\ell\ed\right|\lesssim \delta^{-\ell}$ gives us
\begin{align}\label{I22 estimate}
    |I_{2,2}| &\leq C[u]^2_{C^{0,\gamma}(\Omega)} \int_{\Omega_{\rm far}} |y-x_1|^\gamma |y-x_2|^\gamma  \nonumber\\
    &\qquad \times\left\{\frac{\rho}{|y-\bar{x}|^{d+1}} + \left(\frac{\rho}{\delta|y-\bar{x}|^{d}}+\frac{\rho}{\delta^2|y-\bar{x}|^{d-1}} \right)\1_{{\Omega_{\delta/2}\setminus\Omega_\delta}}\right\}\,\dd y.
\end{align}
Here we recall from \eqref{new, choice of x1, x2} that $\rho=|x_1-x_2| \leq c_
\Omega\delta^\beta$.

The integral $\int_{\Omega_{\rm far}} |y-x_1|^\gamma |y-x_2|^\gamma \frac{\rho}{|y-\bar{x}|^{d+1}} \,\dd y$ can be treated similarly as in \cite[Section~1.2]{note}. In view of the triangle inequality and $0<\gamma<1$, we have
\begin{align*}
    |y-x_1|^\gamma &\leq |y-\bar{x}|^\gamma  + \left(\frac{\rho}{2}\right)^\gamma \\
    &\leq |y-\bar{x}|^\gamma  + \left(\frac{ |y-\bar{x}|}{10}\right)^\gamma \leq 2|y-\bar{x}|^\gamma \qquad\text{for } y \in \Omega_{\rm far},
\end{align*}
and analogously $|y-x_2|^\gamma\leq 2|y-\bar{x}|^\gamma$. Thus, for any $R>0$ so large that $\Omega \subset {\bf B}_R(\bar{x})$, one may estimate
\begin{align}\label{xxx1}
&\int_{\Omega_{\rm far}} |y-x_1|^\gamma |y-x_2|^\gamma \frac{\rho}{|y-\bar{x}|^{d+1}} \,\dd y\nonumber\\
&\qquad\leq C\rho \int_{{\bf B}_R(\bar{x})\setminus {\bf B}_{5\rho}(\bar{x})}\frac{1}{|y-\bar{x}|^{d+1-2\gamma}}\,\dd y \nonumber\\
&\qquad=C(d, \gamma) \rho \big[-s^{-1+2\gamma}\Big]^R_{5\rho}  \nonumber\\
&\qquad\leq C(d,\gamma) \rho^{2\gamma},
\end{align}
by noting $0<\gamma<1/2$.

Now we proceed to the control for  $$\int_{\Omega_{\delta/2}\setminus\Omega_\delta} |y-x_1|^\gamma |y-x_2|^\gamma \left\{\frac{\rho}{\delta|y-\bar{x}|^{d}}+\frac{\rho}{\delta^2|y-\bar{x}|^{d-1}}\right\} \,\dd y.$$ 
Notice here that $\overline{\Omega_{\delta/2}\setminus\Omega_\delta} \subset \Omega_{\rm far}$. As in the previous paragraph, this is bounded by
\begin{align*}
C\int_{\Omega_{\delta/2}\setminus\Omega_\delta} \left\{\frac{\rho}{\delta|y-\bar{x}|^{d-2\gamma}}+\frac{\rho}{\delta^2|y-\bar{x}|^{d-1-2\gamma}}\right\} \,\dd y,
\end{align*}
where $C$ is a universal constant. Then, making use of $|y-\bar{x}| \geq 5\rho$ (recall the definition of $\Omega_{\rm far}$) and ${\rm Volume}\left(\Omega_{\delta/2}\setminus\Omega_\delta\right)\leq C(d)\cdot\delta^{d}$, we may further bound the above expression by 
\begin{align*}
C(d)\rho^{2\gamma}\left\{ \rho^{-d+1} \cdot\delta^{d-1} + \rho^{-d+2} \cdot \delta^{d-2} \right\}.
\end{align*}
As  $\rho \geq c_\Omega\delta$ by condition~\eqref{new, choice of x1, x2}, this is less than or equal to $$C(d,\Omega)\rho^{2\gamma}.$$ %since $d\geq 3$ and $\beta=\beta(d,\gamma) \in ]0,1[$. 

Therefore, we conclude that 
\begin{small}
\begin{equation}\label{xxx2}
\int_{\Omega_{\delta/2}\setminus\Omega_\delta} |y-x_1|^\gamma |y-x_2|^\gamma \left\{\frac{\rho}{\delta|y-\bar{x}|^{d}}+\frac{\rho}{\delta^2|y-\bar{x}|^{d-1}}\right\} \,\dd y \leq C(d,\Omega)\rho^{2\gamma}.
\end{equation}
\end{small}

\smallskip
\noindent
{\bf 5.} We complete the proof of Lemma~\ref{lemma: I2} by putting together the estimates obtained in \eqref{I21 estimate}, \eqref{I22 estimate}, \eqref{xxx1}, and \eqref{xxx2} respectively.  \end{proof}

\subsection{Proof of Theorem~\ref{thm}} \label{subsec: z} From \eqref{pbad key eq} and Lemmas~\ref{lemma: I1} $\&$ \ref{lemma: I2}, we deduce
\begin{align*}
\left|\pbad(x_1) - \pbad(x_2)\right| \leq C|x_1-x_2|^{2\gamma},
\end{align*}
\emph{provided that the condition~\eqref{new, choice of x1, x2} is verified.} That is, the above estimate holds when $[x_1, x_2] \subset \Omega_{10 \cdot \delta^\beta}$ and $c_\Omega\delta\leq  |x_1-x_2| \leq c_\Omega\delta^\beta$, where  $0<\beta<1$ is a fixed number given in \eqref{3.8}, and $c_\Omega>0$ is geometrical. The constant may be chosen as $C=C(d,\gamma) \cdot C_{\rm geom} \cdot \|u\|_{C^{0,\gamma}(\Omega)}^2.$ Then, in view of the definition of $\wp$ (see \eqref{wp, def}), we have $\wp=p$ in $\Omega \setminus {\bf supp}(\pd)$, where $p$ is the pressure we need to estimate. But ${\bf supp}(\pd) \subset \Omega \setminus \Omega_{2\delta}$ while $[x_1, x_2] \subset \Omega_{10\cdot\delta^\beta}$, so $x_1, x_2 \notin {\bf supp}(\pd)$. We thus obtain the desired estimate \eqref{thm, est} in Theorem~\ref{thm}, for those $x_1$, $x_2$ satisfying the condition~\eqref{new, choice of x1, x2}. 

Now, for any given $x_1, x_2 \in \Omega$, we set $$\kappa := \min\Big\{{\rm dist}(x_1, \p\Omega), {\rm dist}(x_2, \p\Omega)\Big\}>0.$$ %Then $x_1, x_2 \in \Omega_{a\kappa}$ for any $a \in ]0,1[$. 
Let $c_\Omega$ be the geometrical constant as in the condition~\eqref{new, choice of x1, x2}, and consider the two separate cases below:

\smallskip
\noindent
\underline{Case~1: $\rho=|x_1-x_2| \geq c_\Omega(12^{-1}\kappa)^{\frac{1}{\beta}}$ or $\rho \geq \left(c_\Omega\right)^{-\frac{\beta}{1-\beta}}$.} If the former holds, we deduce from the estimate~\eqref{p, C0, June25} (in fact, we only need the $C^0$-bound for $p$; see the comments ensuing this estimate) that
\begin{align}\label{big, June25}
\frac{\left|p(x_1)-p(x_2)\right|}{|x_1-x_2|^{2\gamma}} &\leq 2\cdot 12^{\frac{2\gamma}{\beta}} \|p\|_{C^0(\Omega)}\left(c_\Omega\right)^{-2\gamma}\kappa^{-\frac{2\gamma}{\beta}}\nonumber\\
&\leq C(d,\gamma,\Omega)\|u\|^2_{C^{0,\gamma}(\Omega)} \kappa^{-\frac{2\gamma}{\beta}}.
\end{align}
If the latter holds, then the analogous argument yields that
\begin{align}\label{big,, June25}
\frac{\left|p(x_1)-p(x_2)\right|}{|x_1-x_2|^{2\gamma}} &\leq C(d,\gamma,\Omega)\|u\|^2_{C^{0,\gamma}(\Omega)}.
\end{align}

\smallskip
\noindent
\underline{Case~2: $\rho=|x_1-x_2|< \min\left\{c_\Omega(12^{-1}\kappa)^{\frac{1}{\beta}},\,  \left(c_\Omega\right)^{-\frac{\beta}{1-\beta}}\right\}$.} In this case, taking $$\delta := \frac\rho{c_\Omega}$$ leads to $\kappa > 12\delta^\beta$. On the other hand,  $\rho <  \left(c_\Omega\right)^{-\frac{\beta}{1-\beta}}$ implies that $\rho \leq \delta^\beta$. Hence, we obtain $[x_1, x_2] \subset \Omega_{10\delta^\beta}$, by replacing $c_\Omega$ with a larger geometric constant if necessary.

In view of the above paragraph, the condition~\eqref{new, choice of x1, x2} is verified in this case. Thus, by the arguments at the beginning of this subsection, one has that
\begin{align}\label{small, June25}
\left|p(x_1) - p(x_2)\right| \leq C(d,\gamma,\Omega) \|u\|_{C^{0,\gamma}(\Omega)}^2 |x_1-x_2|^{2\gamma}.
\end{align}
In equations~\eqref{big, June25}, \eqref{big,, June25}, and \eqref{small, June25} above, all the constants $C(d,\gamma,\Omega)$ depends only on $d$, $\gamma$, and the $C^2$-geometry of $\Omega$. Putting these estimates together, we obtain
\begin{align}
\frac{\left|p(x_1)-p(x_2)\right|}{|x_1-x_2|^{2\gamma}} \leq C(d,\gamma,\Omega)\|u\|^2_{C^{0,\gamma}(\Omega)} \left(1+\kappa^{-\frac{2\gamma}{\beta}}\right).
\end{align}

This completes the proof of Theorem~\ref{thm}. We also obtain Remark~\ref{remark: constant} by keeping track of the dependence of constants in~\eqref{big, June25}, \eqref{big,, June25}, and \eqref{small, June25}.

\section{Concluding remarks}\label{sec: concl}

In this note, we have given an alternative, elementary proof for the interior double H\"{o}lder regularity of the Euler pressure $p$ in bounded $C^2$-domains of dimension greater than or equal to $3$. We have also obtained refined characterisations for the behaviour of the quantity $|p(x_1)-p(x_2)|/|x_1-x_2|^{2\gamma}$ as $x_1, x_2$ approaches the boundary $\po$ for $\gamma \in ]0,1/2[$.  

We conclude with the following remarks.

\subsection{Two-dimensional case}
The proof of Theorem~\ref{thm} above requires $d\geq 3$. In dimension 2, the Neumann Green function satisfies $|D^\ell\kn(z)|\lesssim |D^\ell\log\,z|$ for $\ell\in\{0,1,2\}$, as opposed to the bound~\eqref{kernel estimate}. An adaptation of the arguments in this note should lead to the same statement of Theorem~\ref{thm} with $d=2$.

\begin{comment}
\subsection{Regularity up to the boundary?} Theorem~\ref{thm} establishes the double H\"{o}lder regularity of $p$ in the interior, namely that $p \in C^{0,2\gamma}(\Omega)$, whenever $u \in C^{0,\gamma}(\Omega)$ for bounded $C^2$-domain $\Omega \subset \R^{d\geq 3}$. It has also been established in De Rosa, Latocca and Stefani \cite[Theorem~1.1]{dls1} that if $\Omega$ is a $C^{2,\alpha}$-domain for any $\alpha>0$ and if $u \in C^{0,\gamma}({\Omega};\R^d)$, then $p$ is continuous up to the boundary; \emph{i.e.}, $p \in C^{0}(\overline{\Omega})$.

In addition to the interior double H\"{o}lder regularity of $p$, we have further obtained that $p$ has no ``small-scale double H\"{o}lder creations'' near the boundary in the following sense:
\begin{align*}
\limsup_{\kappa \searrow 0} \left( \sup\left\{ \frac{|p(x_1)-p(x_2)|}{|x_1-x_2|^{2\gamma}}:\, |x_1-x_2| \lesssim \min_{i\in\{1,2\}}\dist(x_i,\po)\approx \kappa \right\}\right) <\infty,
\end{align*}
where the constants involved in $\lesssim$ and $\approx$ are purely geometrical. 

At the moment we are unable to obtain the analogous estimate uniformly in $\kappa$  for large $|x_1-x_2|$  as $\kappa\searrow 0$. But in this case an  $\mathcal{O}\left(\kappa^{-1+2\gamma}\right)$-upper bound for the blowup rate of the double H\"{o}lder norm of $p$ has been obtained.
\end{comment}

\subsection{Unbounded domains} 
It would be interesting to investigate the case for an unbounded domain $\Omega\subset\R^d$. We expect that the double H\"{o}lder regularity for $p$ remains valid for unbounded domains \emph{with bounded geometry} in the sense of Schick \cite{sch}: 
\begin{itemize}
    \item 
     the boundary $\po$ has a uniform geodesic $r$-collar for some $r>0$;
     \item 
     the boundary $\po$ has positive injectivity radius;
     \item 
     $\Omega_{r/3}$ in the interior has positive injectivity radius;
     \item 
     the second fundamental form $\two$ of $\po$ has uniform $C^\ell$-bounds for every $\ell \in \mathbf{N}$. 
\end{itemize}
See also Disconzi, Shao and Simonett \cite{dss} for an equivalent characterisation.

In fact, we expect that this is valid when the last condition is replaced by the weaker condition: ``\emph{$\two$ has a uniform $C^0$-bounds over $\po$}''. In this case, one may say that the unbounded domain $\Omega$ has \emph{bounded $C^2$-geometry}. 

The proof should follow from the existence of a ``tame'' partition of unity consisting of boundary charts of comparable diameters, in which the second fundamental forms has bounded $C^0$-norms all comparable to each other. See Ammann, Gro{\ss}e and Nistor \cite{agn} for details.

\bigskip
\noindent
{\bf Acknowledgement}.
Both authors thank the anonymous referees for their careful reading and constructive suggestions. We are indebted to Prof.~Claude Bardos, Mr.~Daniel Boutros, and Prof. Edriss Titi for pointing out a flawed argument in an earlier version of the manuscript. SL also thanks Professors Linhan Li and Lihe Wang for insightful discussions on Green functions.

\bigskip
\noindent
{\bf Funding Declaration}.
This research was partially supported by NSFC under Grant No. 12331008, and the Shanghai Frontier Research Institute for Modern Analysis.  
The research of SL was also partially supported by NSFC under Grant Nos. 12201399 and 12411530065, Young Elite Scientists Sponsorship Program by CAST  2023QNRC001, National Key Research and Development Programs 2023YFA1010900 and 2024YFA1014900, and Shanghai Rising-Star Project. The research of YGW was also partially supported by NSFC under Grant Nos. 12171317, 12250710674 and 12161141004, and Shanghai Municipal Education Commission under Grant No. 2021-01-07-00-02-E00087.

\bigskip
\noindent
{\bf Author Contribution Declaration}. SL and YGW wrote the main manuscript text and reviewed the manuscript.

\bigskip
\noindent
{\bf Competing interests statement}. 
Both authors declare that there is no conflict of interest.

\bigskip
\noindent
{\bf Data availability statement}.
Our manuscript has no associated data.

% \clearpage

%\bibliographystyle{abbrv}
%\bibliography{refs}
\end{document}